\documentclass[12pt]{amsart}
\usepackage{amsthm,amsmath,amssymb,amsfonts}

\newcommand{\ga}{\gamma}
\newcommand{\de}{\delta}
\newcommand{\ka}{\kappa}
\newcommand{\Ga}{\Gamma}
\newcommand{\De}{\Delta}
\newcommand{\Th}{\Theta}
\newcommand{\Om}{\Omega}

\newcommand{\Ag}{\mathcal{A}_{g}}
\newcommand{\M}{\mathcal{M}}
\newcommand{\Mg}{\M_g}
\newcommand{\Mgn}{\M_{g,n}}
\newcommand{\Mgg}{\M_{g,2g+3}}
\newcommand{\Mct}{\M^{ct}}
\newcommand{\Mrt}{\M^{rt}}
\newcommand{\Mgnct}{\M^{ct}_{g,n}}
\newcommand{\Mgnrt}{\M^{rt}_{g,n}}
\newcommand{\Mggrt}{\M^{rt}_{g,2g+3}}

\newcommand{\oM}{\overline{\mathcal{M}}}
\newcommand{\oMgn}{\overline{\mathcal{M}}_{g,n}}

\newcommand{\calO}{\mathcal{O}}
\newcommand{\Xg}{\mathcal{X}_g}
\newcommand{\Zg}{\mathcal{Z}_{g}}
\newcommand{\ZgA}{\mathcal{Z}_{g,A}}
\newcommand{\ZZ}{\mathbb{Z}}
\newcommand{\QQ}{\mathbb{Q}}
\newcommand{\PP}{\mathbb{P}}
\DeclareMathOperator{\Jac}{Jac}
\DeclareMathOperator{\Aut}{Aut}

\theoremstyle{plain}
\newtheorem{theorem}{Theorem}
\newtheorem{lemma}[theorem]{Lemma}

\newtheorem{corollary}[theorem]{Corollary}

\newtheorem{claim}{Claim}

\theoremstyle{definition}

\newtheorem{remark}[theorem]{Remark}

\begin{document}
\title[Powers of the theta divisor and vanishing]{Powers of the theta divisor and relations in the tautological ring}
\author[E.~Clader]{Emily Clader}
\address{Mathematics Department, San Francisco State University, San Francisco, CA 94132-1722, USA.}
\email{eclader@sfsu.edu}
\thanks{The first author acknowledges the generous support of Dr.~Max R\"ossler, the Walter Haefner Foundation, and the ETH Foundation.  Research of the second author is supported in part by the National Science Foundation under the grants DMS-12-01369 and DMS-15-01265, and by a Simons Fellowship in Mathematics (Simons Foundation grant \#341858 to Samuel Grushevsky)}
\author[S.~Grushevsky]{Samuel Grushevsky}
\address{Mathematics Department, Stony Brook University, Stony Brook, NY 11794-3651, USA.}
\email{sam@math.stonybrook.edu}
\author[F.~Janda]{Felix Janda}
\address{Mathematics Department, University of Michigan, Ann Arbor, MI 48109-1043, USA.}
\email{janda@umich.edu}
\author[D.~Zakharov]{Dmitry Zakharov}
\address{Courant Institute of Mathematical Sciences, New York University, New York, NY 10012-1185, USA.}
\email{dvzakharov@gmail.com}

\begin{abstract}

We show that the vanishing of the $(g+1)$-st power of the theta divisor in the cohomology and Chow rings of the universal abelian variety implies, by pulling back along a collection of Abel--Jacobi maps, the vanishing results in the tautological ring of $\Mgn$ of Looijenga, Ionel, Graber--Vakil, and Faber--Pandharipande. We also show that Pixton's double ramification cycle relations, which generalize the theta vanishing relations and were recently proved by the first and third authors, imply Theorem~$\star$ of Graber and Vakil. Moreover, our proof provides an algorithm for expressing any tautological class on $\oMgn$ of sufficiently high codimension as a tautological class supported on the boundary.

\end{abstract}

\maketitle

\section{Introduction}
The tautological ring $R^*(\Mg)$ is the subring, of either the cohomology or the Chow ring of $\Mg$, generated by the Mumford--Morita--Miller $\kappa$-classes \cite{1983Mumford,1984Morita,1986Miller} defined by
\begin{equation*}
  \ka_i=\pi_*(\psi^{i+1}),
\end{equation*}
where $\pi\colon \M_{g,1} \rightarrow \M_g$ and $\psi = c_1(\omega_{\pi})$.  Faber and Pandharipande \cite{2005FaberPandharipande} gave an elegant extension of this definition to the Deligne--Mumford compactification: the rings $R^*(\oM_{g,n})$ are the smallest system of $\QQ$-subalgebras (either of the cohomology or the Chow ring of $\oM_{g,n}$) closed under pushforward by the gluing morphisms
\[\oM_{g_1, n_1+1} \times \oM_{g_2, n_2+1} \rightarrow \oM_{g_1+g_2, n_1+n_2},\]
\[\oM_{g,n+2} \rightarrow \oM_{g+1,n}\]
and the forgetful morphisms
\[\oM_{g,n+1} \rightarrow \oM_{g,n}.\]
Restricting $R^*(\oM_g)$ to $\M_g$ recovers the $\kappa$-ring $R^*(\M_g)$, while more generally, the restriction of the tautological ring to the open moduli space $\M_{g,n}$ is generated by the $\kappa$-classes together with the $\psi$-classes,
\[\psi_i = c_1(s_i^*\omega_{\pi}),\]
where $\pi\colon \M_{g,n+1} \rightarrow \M_{g,n}$ denotes the forgetful map, and $s_i$ is the section of $\pi$ determined by the $i$-th marked point.

Faber initiated an extensive study of $R^*(\M_g)$ in \cite{1999Faber}, calculating relations among tautological classes by using the Porteous formula and the methods and results of Mumford \cite{1983Mumford}.  Based on his observations, he formulated a striking series of conjectures suggesting that the tautological ring possesses a rich and remarkably well-behaved structure.  In particular, he proposed that $R^*(\M_g)$ vanishes in sufficiently high codimension, a result that was proven by Looijenga and Ionel:
\begin{theorem}[Looijenga \cite{1995Looijenga}, Ionel \cite{2002Ionel}] The tautological ring of $\Mgn$ vanishes in degrees greater than or equal to $g$, as well as in degree $g-1$ when $n=0$.
\label{thm:Fabervanishing}
\end{theorem}

A key geometric insight into these and other properties of the tautological ring was provided by Graber and Vakil in \cite{2005GraberVakil}, in which Theorem~\ref{thm:Fabervanishing} was shown to be a consequence of the following result:

\begin{theorem}[Theorem~$\star$ of \cite{2005GraberVakil}] Any tautological class on $\oMgn$ of codimension $k$ can be represented by a class supported on the locus of curves having at least $k-g+1$ rational components.
\label{thm:star}
\end{theorem}

More generally, Theorem~$\star$ also implies several other properties of the tautological rings, including the analogous vanishing statements to Theorem~\ref{thm:Fabervanishing} for curves with rational tails, or for curves of compact type, and the socle statements for these moduli spaces (that is, the one-dimensionality of the tautological ring in the smallest codimension where it is nontrivial); see \cite[Section 5]{2005GraberVakil}.

A stronger form of Theorem~$\star$ was proved by Faber and Pandharipande:

\begin{theorem}[\cite{2005FaberPandharipande}]
\label{thm:strongstar}
There exists an expression for any codimension $k$ tautological class in terms of tautological classes supported on curves with at least $k-g+1$ rational components.
\end{theorem}

The point of this strengthening of the previous theorem is that the boundary expression is itself tautological. However, implementing the proof of~\cite{2005FaberPandharipande} as an algorithm to compute this expression explicitly seems to be computationally impractical.

Our first result is a new proof of Theorem~\ref{thm:Fabervanishing} from a family of tautological relations on $\M_{g,n}^{ct}$ that we call the $\Theta$-relations. These relations arise by pulling back the universal theta divisor on the universal abelian variety to $\M_{g,n}^{ct}$ under a suitable family of Abel--Jacobi maps, and observing that its $(g+1)$-st power vanishes in the Chow ring. These were discussed in~\cite{2013Hain,2014GrushevskyZakharovB}, and we give the details in Section~2. Thus, we prove:
\begin{theorem}
The $\Theta$-relations on $\M_{g,n}^{ct}$ imply Theorem~\ref{thm:Fabervanishing} in the Chow ring of $\Mgn$.
\label{thm:mainvanishing}
\end{theorem}

The key advantage of this new proof of Theorem~\ref{thm:Fabervanishing} --- aside from the fact that it proceeds by an entirely elementary argument from the $\Theta$-relations --- is that it leads naturally to a constructive proof of Theorem~$\star$ (and hence also implies Theorem~\ref{thm:strongstar}).  Indeed, the $\Theta$-relations can be viewed as the restriction to $\M_{g,n}^{ct}$ of a family of relations called the double ramification cycle relations (see Theorem~\ref{thm:Pixton} below), first conjectured by Pixton and later proved by the first and third authors \cite{2016CladerJanda}.  These relations arise, as we discuss in Section \ref{DRrelations}, from a perspective on the theta divisor via the moduli space of relative stable maps to the projective line.  By carefully tracking the boundary contributions to the double ramification cycle relations, the vanishing in the proof of Theorem~\ref{thm:Fabervanishing} from the $\Theta$-relations is upgraded to an algorithm for computing tautological boundary expressions for tautological classes in degree at least $g$. Implementing this algorithm explicitly on a computer is an entirely tractable task that we intend to carry out in future work.

We summarize the preceding discussion as our main theorem:
\begin{theorem}
Pixton's double ramification cycle relations on $\oM_{g,n}$ imply Theorem~$\star$ (and its strengthening, Theorem~\ref{thm:strongstar}), and moreover yield an algorithm for expressing, in the Chow ring, any tautological class on $\oM_{g,n}$ of codimension at least $g$ as an explicit tautological class supported on the boundary.
\label{thm:mainthm}
\end{theorem}

\smallskip
The paper is organized as follows.  We recall the definition of the $\Theta$-relations and Pixton's double ramification cycle relations in Section \ref{relations}, and in Section \ref{sec:main}, we give the proofs of Theorems~\ref{thm:mainvanishing} and \ref{thm:mainthm}.  We then exemplify our algorithm by using it to compute boundary expressions for $\psi_1$ and $\kappa_1$ on $\oM_{1,1}$ in Section \ref{examples}.  Throughout the paper, we work in the Chow ring with rational coefficients; all results also imply the analogous statements in the cohomology with rational coefficients, since, by definition, the cycle class map is surjective on the tautological ring.

\section{The $\Theta$-relations and Pixton's double ramification cycle relations}
\label{relations}

Fix a genus $g\geq 0$ and an integer $n>0$ such that $2g - 2 + n > 0$, and let
\begin{equation*}
  A:=(a_1,\ldots,a_n),\quad a_i\in \ZZ
\end{equation*}
be a vector satisfying the condition
\begin{equation*}
  \sum_{i=1}^n a_i=0.
\end{equation*}
Define the locus $\ZgA \subset \Mgn$ to consist of marked curves $(C,p_1,\ldots,p_n)$ satisfying
\begin{equation}
  \calO_C\left(\sum_{i=1}^n a_ip_i\right)\cong \calO_C.
\label{eq:DRopen}
\end{equation}
Then $\ZgA$ is of pure codimension $g$ in $\Mgn$. Eliashberg posed the question of computing the class $[\ZgA]$ of $\ZgA$, in the cohomology or the Chow ring of $\Mgn$, and of defining and computing an extension of $[\ZgA]$ to $\oMgn$.

\subsection{The $\Th$-relations and moduli of abelian varieties}
\label{subsec:Thetarelations}
One approach to Eliashberg's problem, introduced by Hain in \cite{2013Hain}, is to consider \eqref{eq:DRopen} as a relation on the Jacobian variety of the curve $C$, and to vary it in moduli. Let $\Ag$ denote the moduli space of principally polarized abelian varieties, and define the Abel--Jacobi map $s_A\colon\Mgn\to\Xg$ to the universal abelian variety $\Xg=\left\lbrace([A],z)\;|\;[A]\in \Ag, z\in A\right\rbrace$ by the formula
\begin{equation}\label{eq:AJ}
  s_A(C,p_1,\ldots,p_n):=(\Jac^0_C,\calO_C(a_1p_1+\ldots+a_np_n)).
\end{equation}
The locus $\ZgA$ is then easily seen to be the preimage of the zero section of $\Xg$ under $s_A$:
\begin{equation*}
  \ZgA:=s_A^{-1}\Zg,\quad {\rm where}\quad \Zg:=\left\lbrace([A],0)\;|\;[A]\in \Ag\right\rbrace\subset \Xg.
\end{equation*}

A stable curve is of compact type if and only if its Jacobian is an abelian variety, and the Abel--Jacobi map $s_A$ naturally extends to the moduli space of curves of compact type $\Mgnct\subset \oMgn$ (see formula (14) in \cite{2014GrushevskyZakharovA} or Section 0.2.3 in \cite{2016JandaPandharipandePixtonZvonkine}).  To define this extension, let $(C,p_1,\ldots,p_n)$ be a stable marked curve of compact type, let $\widetilde{C}$ be the normalization, let $\widetilde{C}=C_1\sqcup \ldots\sqcup C_N$ be the decomposition of $C$ into irreducible components, and for each $j$, let $q_{jk}\in C_j$ be the preimages of the nodes. There is a unique way to assign weights $a_{jk}$ to the nodes such that the weights at a pair of matching nodes sum to zero, and on each $C_j$, we have
\[\sum_{i \; : \;  p_i \in C_j} a_i + \sum_k a_{jk} = 0.\]
From here, since $\Jac (C)=\prod_{j=1}^N \Jac(C_j)$, one can define the extension $s_A\colon\Mgnct\to \Xg$ by the following formula:
\begin{equation}
s_A(C,p_1,\ldots,p_n)=\left(\Jac(C),\prod_{j=1}^N \calO_{C_j}\left(\sum_{i:p_i\in C_j}^n a_ip_i+\sum_k a_{jk}q_{jk}\right)\right).
\label{eq:AJcompact}
\end{equation}
The closure of $\ZgA$ in $\Mgnct$ is contained in $s_A^{-1}(\Zg)$, and  $s_A^{-1}(\Zg)\cap \Mgn=\ZgA$.  However, $s_A^{-1}(\Zg)$ has components of excessive dimension, coming from curves having a rational tail. To account for these, we consider not the preimage but the pullback
\begin{equation} \label{eq:Rct}
  R^{ct}_{g,A}:=s_A^*[\Zg],
\end{equation}
either in the cohomology ring or in the Chow ring of $\Mgnct$.

In \cite{2013Hain}, Hain computed the class $R^{ct}_{g,A}$ in cohomology, using Hodge-theoretic techniques. Hain's calculations were simplified and extended to the Chow ring by the second and fourth authors in \cite{2014GrushevskyZakharovB}. The main idea is to use the following theorem, which can be deduced from results of Deninger and Murre \cite[Cor.~2.22]{1991DeningerMurre} by applying the Fourier--Mukai transform to theta divisor; see~\cite[Exercise~13.2]{2012EdixhovenvanderGeerMoonen} or~\cite[Sec.~16.4, Exercise~16.8.1]{2004BirkenhakeLange} for further details.
\begin{theorem}\label{thm:thetact}
Let $\Th\in CH^1(\Xg)$ denote the universal symmetric theta divisor trivialized along the zero section. Then
\begin{enumerate}
\item $[\Zg]=\dfrac{\Th^g}{g!}$ in $CH^g(\Xg)$,
\item $\Th^{g+1}=0$ in $CH^{g+1}(\Xg).$
\end{enumerate}
\end{theorem}

To compute $R_{g,A}^{ct}$, it thus suffices to compute the pullback of the divisor $\Th$, which is a standard calculation using test curves. In addition, for every $A$ we get a relation $[s_A^*\Th]^{g+1}=0$  in $CH^{g+1}(\Mgnct)$. We summarize the results of \cite{2013Hain} and \cite{2014GrushevskyZakharovB} in the following theorem, in which $[n]$ denotes the set $\{1, \ldots, n\}$:

\begin{theorem}[$\Th$-relations] \label{thm:Theta} The pullback of $\Th$ to $\Mgnct$ along the Abel--Jacobi map \eqref{eq:AJcompact} is equal to
\begin{equation}\label{eq:Theta}
  s_A^*\Th=-\frac{1}{4}\sum_{h=0}^g \sum_{P\subset [n]}a_P^2 \delta_h^P,
\end{equation}
where $a_P=\sum_{i\in P}a_i$ for any $P\subset [n]$. Here, $\delta_h^P$ is the class of the closure of the locus of curves having an irreducible component of genus $h$ containing the marked points indexed by $P$ and an irreducible component of genus $g-h$ containing the remaining marked points. In the cases when the resulting curve is not stable, we set $\delta_0^{\{i\}}=\delta_g^{[n]\backslash\{i\}}=-\psi_i$ and $\delta_0^{\emptyset}=\delta_g^{[n]}=0$.

Formula \eqref{eq:Theta} and Theorem \ref{thm:thetact} imply that the pullback of the zero section is given by the following formula:
\begin{equation}\label{eq:Thetag}
R_{g,A}^{ct}=\frac{1}{g!}[s_A^*\Th]^{g}\in CH^{g}(\Mgnct),
\end{equation}
and that following relation holds:
\begin{equation}\label{eq:Thetag+1}
  [s_A^*\Th]^{g+1}=0\in  CH^{g+1}(\Mgnct).
\end{equation}
\end{theorem}

In \cite{2014GrushevskyZakharovA}, the second and fourth authors pushed this method further, proving analogues of Theorems~\ref{thm:thetact} and \ref{thm:Theta} for an extension of the Abel--Jacobi map to $s_A\colon \oMgn^o \to \Xg^{part}$, where $\oMgn^o$ is the moduli space of curves with at most one non-separating node and $\Xg^{part}$ is the universal family over Mumford's partial compactification of $\Ag$, which parameterizes semiabelian varieties of torus rank at most one.  Attempting to proceed any deeper into the boundary, however, one encounters two technical difficulties. First, the extension of the formula $[\Zg]=\Th^g/g!$ to $\Xg^{part}$ obtained in \cite{2014GrushevskyZakharovA} is already quite complicated, and the deeper boundary strata of $\Xg$ have an increasingly complex combinatorial structure, so the corresponding calculations do not seem combinatorially manageable at this point. In addition, on boundary strata not contained in  $\oMgn^o$ the Abel--Jacobi map does not extend to a morphism, and its indeterminacy locus needs to be resolved. The question of extending the Abel--Jacobi map \eqref{eq:AJcompact} to various compactifications of the universal Jacobian variety and resolving the singularities of the Abel--Jacobi map has been considered in a number of recent papers (see  \cite{2015Dudin, 2014Holmes, 2015KassPagani, 2011Melo,  2014MeloRapagnettaViviani, 2016Melo}).

\subsection{The double ramification cycle and Pixton's relations}
\label{DRrelations}

An alternative way to extend $\ZgA$ to $\oM_{g,n}$ is given by the double ramification cycle.  To motivate the definition, we observe that the locus $\ZgA\subset\Mgn$ of marked curves satisfying condition \eqref{eq:DRopen} has an equivalent definition as the locus of $(C,p_1,\ldots,p_n)$ admitting a ramified cover $f\colon C\to \PP^1$ such that
\begin{itemize}
\item $f^{-1}(0)=\left\lbrace p_i\; | \; a_i>0\right\rbrace$ and $f^{-1}(\infty)=\left\lbrace p_i \;|\;a_i<0\right\rbrace$.
\item The ramification profiles of $f$ over $0$ and $\infty$ are $\mu=\left\lbrace a_i\;|\;a_i>0\right\rbrace$ and $\nu=\left\lbrace |a_i|\;|\;a_i<0\right\rbrace$, respectively.
\end{itemize}
(No condition is imposed on the marked points $p_i$ with $a_i = 0$.) To extend $\ZgA$ to $\oMgn$, then, one can compactify the space of such ramified covers, allowing both $C$ and the target $\PP^1$ to degenerate. The resulting object is known as the moduli space of rubber relative stable maps to $\PP^1$ and is denoted $\oM_{g,n_0}(\PP^1;\mu,\nu)^{\sim}$, where $n_0:=\#\left\lbrace i\; | \;a_i=0\right\rbrace$.

For more details on the moduli space of rubber relative stable maps, see \cite{2005FaberPandharipande}.  The crucial property required here is that it admits a virtual fundamental class (constructed algebraically by Jun Li \cite{2001Li}):
\[[\oM_{g,n_0}(\PP^1;\mu,\nu)^{\sim}]^{\text{vir}} \in CH_{\text{vdim}}(\oM_{g,n_0}(\PP^1;\mu,\nu)^{\sim}).\]
From here, one defines the double ramification cycle by pushforward
\begin{equation*}
  R_{g,A}:=\tau_*[\oM_{g,n_0}(\PP^1;\mu,\nu)^{\sim}]^{\mathrm{vir}}\in CH^g(\oMgn)
\end{equation*}
along the natural forgetful morphism $\tau\colon\oM_{g,n_0}(\PP^1;\mu,\nu)^{\sim}\to\oMgn$.

The restriction of $R_{g,A}$ to $\Mgn$ is the class of $\ZgA$, essentially by definition. Marcus and Wise \cite{2013MarcusWise} proved, moreover, that the restriction of $R_{g,A}$ to $\Mgnct$ is equal to the class $R^{ct}_{g,A}$ defined by \eqref{eq:Rct}, and in \cite{2012CavalieriMarcusWise}, Cavalieri, Marcus, and Wise independently computed the restriction of $R_{g,A}$ to the moduli space $\Mgnrt$ of curves with rational tails, obtaining the same result as in Theorem \ref{thm:Theta}.

Faber and Pandharipande proved in \cite{2005FaberPandharipande} that the class $R_{g,A}$ is tautological, and provided a method to calculate it in principle.  The computational difficulties of their method, however, are too great to obtain an explicit formula.  This situation was remedied by a conjecture of Pixton, which proposed not only a formula for $R_{g,A}$ (generalizing equation \eqref{eq:Thetag} of Theorem \ref{thm:Theta}) but also a generalization of the $\Theta$-relations (equation \eqref{eq:Thetag+1} of Theorem \ref{thm:Theta}) to all of $\oM_{g,n}$.

The basic idea of Pixton's conjecture, which we explain further in Section \ref{subsec:Pixton} below, is to view the terms $[s_A^*\Theta]^g$ and $[s_A^*\Theta]^{g+1}$ appearing in Theorem~\ref{thm:Theta} as appropriate multiples of the parts in degree $g$ and $g+1$ of the mixed-degree class $\exp [s_A^*\Theta]\in CH^*(\Mgnct)$.  The expression \eqref{eq:Theta} can be packaged into an elegant formula for $\exp [s_A^*\Theta]$ as a graph sum, and Pixton used an ingenious modification of this graph sum to extend $\exp [s_A^*\Theta]$ to a class $\Omega_{g,A} \in CH^*(\oMgn)$.  Generalizing both statements of Theorem~\ref{thm:Theta}, then, he conjectured that $\Omega_{g,A}$ coincides with $R_{g,A}$ in codimension $g$ and vanishes in higher codimension.

Both parts of Pixton's conjecture have recently been proven:

\begin{theorem}[\cite{2016JandaPandharipandePixtonZvonkine, 2016CladerJanda}]
Let $[\bullet]_d$ denote the degree-$d$ part of a mixed-degree class in $CH^*(\oMgn)$. Then the class $\Om_{g,A}$ satisfies the following:
\begin{enumerate}
\item $[\Om_{g,A}]_g=R_{g,A}$.
\item $[\Om_{g,A}]_d=0$ for $d>g$.
\end{enumerate}
\label{thm:Pixton}
\end{theorem}

The second of these statements is what we refer to as ``Pixton's double ramification cycle relations."

\begin{remark} The restriction of the class $\Om_{g,A}$ to the moduli space $\oMgn^o$ of \cite{2014GrushevskyZakharovA} is equal to the pullback of the zero section $\Zg^{part}$ of the partial compactification of $\Xg^{part}$ under the Abel--Jacobi map $s_A$.
\end{remark}

We give the explicit formula for $\Om_{g,A}$ as equation \eqref{eq:pixton} in Section \ref{subsec:Pixton}. The calculations of Section \ref{sec:main}, in which we prove Theorems~\ref{thm:mainvanishing} and \ref{thm:mainthm}, do not require the full formula for $\Om_{g,A}$ but only the existence of an extension of the $\Theta$-relations.

\section{Proofs of vanishing theorems}

\label{sec:main}

We now turn to the proof of the vanishing of the tautological ring $R^k(\Mgn)$ in degree $k \geq g - \delta_{0n}$, and of Graber--Vakil's Theorem~$\star$, via explicit computation from the relations discussed in the previous section.  More precisely, what we need are the $\Theta$-relations (equation~\eqref{eq:Thetag+1}), the fact that these are polynomial relations of degree $2g+2$ in the variables $a_i$, and the fact that for each tuple $A$ there exists a relation
\begin{equation}
\label{eq:thetaRHS}
[s_A^*\Theta]^{g+1}=D_{A,g+1}\in CH^{g+1}(\oMgn),
\end{equation}
where $s_A^*\Theta$ on $\oMgn$ is defined by \eqref{eq:Theta} and $D_{A,g+1}$ is a tautological class supported away from $\Mgnct$; this last fact follows immediately from Pixton's double ramification cycle relations.  We remark that it is conjectured, though not definitively established (see \cite{PixDR2} or \cite[Lemma 2.1]{2016CladerJanda}), that the class $[\Omega_{g,A}]_{g+1}$ is polynomial in the variables $a_i$, but this is not necessary for our results.

\subsection{Some properties of tautological classes}\label{subsec:technical}

Several formulas involving the tautological classes on $\oMgn$ will be useful later.

Let $\pi_{n+1}\colon\oM_{g,n+1}\to\oMgn$ be the map that forgets the last marked point. The pullbacks of $\psi$- and $\ka$-classes under it are as follows (see \cite[Section 1]{1996ArbarelloCornalba}):
\begin{equation}
\label{eq:pullback}
  \pi_{n+1}^*\psi_i=\psi_i-\de_0^{\{i,n+1\}},\quad \pi_{n+1}^*\ka_a=\ka_a-\psi_{n+1}^a,
\end{equation}
where $\de_0^{\{i,n+1\}}$ is defined as in Theorem~\ref{thm:Theta}.  The pushforward of a monomial in the $\psi$-classes is equal to
\begin{equation*}
  \pi_{n+1,*}[\psi_1^{k_1}\cdots\psi_n^{k_n}\psi_{n+1}^{k_{n+1}+1}]=\psi_1^{k_1}\cdots\psi_n^{k_n}\ka_{k_{n+1}}.
\end{equation*}

More generally, suppose $n>m$.  Then the pushforward of a monomial in the $\psi$-classes from $\oMgn$ to $\oM_{g,m}$ is given by the following formula, originally due to Faber:
\begin{equation}\label{eq:pushka}
  (\pi_{m+1}\circ\cdots\circ\pi_n)_*[\psi_1^{k_1}\cdots\psi_m^{k_m}\psi_{m+1}^{k_{m+1}+1}\cdots\psi_n^{k_n+1}]=\psi_1^{k_1}\cdots\psi_m^{k_m}R(k_{m+1},\ldots,k_n).
\end{equation}
Here, $R(k_{m+1},\ldots, k_n)$ is a polynomial in the $\ka$-classes,
\begin{equation*}
  R(k_{m+1}, \ldots,k_n)= \sum_{\sigma\in S_{n-m}} \ka_{\sigma},
\end{equation*}
where $\ka_{\sigma}$ is defined as follows:
Given a permutation $\sigma\in S_{n-m}$, write it as a product $\sigma=\alpha_1\cdots\alpha_{\nu(\sigma)}$ of disjoint cycles, and for a cycle $\alpha$, define $|\alpha|=\sum_{i\in\alpha} k_{m + i}$ to be the sum of those among $k_{m+1},\ldots,k_n$ that are permuted by $\alpha$. Then
\begin{equation*}
  \ka_{\sigma}:=\ka_{|\alpha_1|}\cdots\ka_{|\alpha_{\nu(\sigma)}|}.
\end{equation*}
Note that
\begin{equation}
\label{eq:observation}
R(k_{m+1}, \ldots,k_n)=\ka_{k_{m+1}}\cdots \ka_{k_n}+ (\text{monomials in fewer than $n-m$ $\ka$-classes}),
\end{equation}
where the first term on the right-hand side corresponds to the trivial permutation in $S_{n-m}$. In particular, we note that $R(k,0,\ldots,0)$ is a nonzero multiple of $\ka_k$ and $R(0,\ldots,0)$ is a nonzero constant (recall that $\ka_0=2g-2+n$).

We will also use the following simple lemma:
\begin{lemma} \label{le:vanish}
Let $1\leq i_1,\ldots,i_k\leq n$ be integers, and let $I\subset[n]$. Then
\begin{equation}\label{eq:vanish}
 \de_0^I\psi_{i_1}\cdots\psi_{i_k}=0
\end{equation}
on $\oMgn$ whenever $\#\left\lbrace j|i_j\in I\right\rbrace\geq \# I-1$.
\end{lemma}
\begin{proof}
Indeed, $\de_0^I$ is the class of the boundary divisor $\De_0^I$, which is the image under the gluing map of the product $\oM_{0,\# I +1}\times \oM_{g,n-\# I+1}$. When $i_j\in I$, it is clear that the class $\psi_{i_j}$ on $\oMgn$ restricts on $\De_0^I$ to the $\psi$-class of the corresponding point on $\oM_{0,\# I+1}$. Hence any class of the form~\eqref{eq:vanish} vanishes if $\#\left\lbrace j|i_j\in I\right\rbrace\geq \#I-1=\dim \oM_{0,\# I+1}+1$.
\end{proof}

Finally, we recall the low genus topological recursion relations, expressing the divisors $\psi_i$ and $\ka_1$ as boundary divisors (see Theorem 2.2 in \cite{1998ArbarelloCornalba}):
\begin{equation}\label{eq:M11}
\ka_1=\psi_1=\frac{1}{12}\delta_{irr} \in CH^1(\oM_{1,1}),
\end{equation}
\begin{equation}\label{eq:TRR}
\psi_i=\sum_{\substack{i\in I\subset [n]\\ a,b \notin I}}\delta_0^I\in CH^1(\oM_{0,n}).
\end{equation}
Here, $\delta_{irr}$ is the class of the locus of curves that have a non-separating node, and in equation~\eqref{eq:TRR}, $a$ and $b$ are any two fixed elements of $[n] \setminus \{i\}$.

\subsection{Additive generators of the tautological ring}
\label{sec:strata}

An explicit set of additive generators for the tautological ring can be defined via the strata algebra (see \cite{2003GraberPandharipande} or \cite{2013Pixton}).  Specifically, recall that a stable graphs $\Ga=(V,H,g,p,\iota)$ consist of the following data:
\begin{enumerate}
\item A set of vertices $V$ equipped with a genus function $g\colon V\to \mathbb{Z}_{\geq 0}$.

\item A set of half-edges $H$ equipped with a vertex assignment $p\colon H\to V$ and an involution $\iota\colon H\to H$.

\end{enumerate}
We define the set of edges $E$ of $\Ga$ to be the set of orbits of $\iota$ that are of cardinality $2$, and the set of legs $L$ of $\Ga$ to be the set of fixed points of $\iota$; the pair $(V,E)$ is then an ordinary graph. The valence of a vertex $v\in V$ is defined as $n(v)=\#  p^{-1}(v)$.  The condition of stability is that $\Gamma$ must be connected and that each vertex $v$ must satisfy $2g(v) -2+n(v)>0$.

We define the genus of a stable graph $\Ga$ to be
\begin{equation}\label{eq:genus}
g(\Ga)=h^1(\Ga)+\sum_{v\in V}g(v),
\end{equation}
where $h^1(\Ga)=\# E-\# V+1$.
An automorphism of a stable graph is a permutation on both $V$ and $H$ that preserves the incidence relations and acts as the identity on legs.

Given a stable curve $C$ of genus $g$ with $n$ marked points, its dual graph is a stable graph of genus $g$ with $n$ legs. For a stable graph $\Ga$ of genus $g$ with $n$ legs, let
\begin{equation*}
\oM_{\Ga}:=\prod_{v\in V}\oM_{g(v),n(v)}.
\end{equation*}
Then there is a canonical gluing morphism
\begin{equation}\label{eq:xi}
\xi_{\Ga}\colon\oM_{\Ga}\to \oMgn,
\end{equation}
whose image is the locus in $\oMgn$ in which the generic point corresponds to a curve with stable graph $\Ga$.

Let $\Ga$ be a stable graph. For each $v\in V$, let $\lbrace x_i[v]\rbrace_{i>0}$ and $\lbrace y[h]\rbrace_{h\in p^{-1}(v)}$ be sets of positive integers.  Associated to each such choice of integers, there is a basic class
\begin{equation*}
\ga_v=\prod_{i>0}\ka_i^{x_i[v]}\prod_{h\in p^{-1}(v)}
\psi_h^{y[h]}\in CH^{d(\ga_v)}(\oM_{\Ga}),
\end{equation*}
having degree
\begin{equation*}
d(\ga_v)=\sum_{i>0}ix_i[v]+\sum_{h\in p^{-1}(v)} y[h].
\end{equation*}
We define
\begin{equation}\label{eq:gamma}
\ga=\prod_{v\in V}\ga_v \in CH^{d(\ga)}(\oM_{\Ga}),
\end{equation}
whose degree is $d(\ga):=\sum_{v\in V}d(\ga_v)$. The pair $(\Ga,\ga)$ defines a tautological class $\xi_{\Ga*}(\ga)$ of degree $d(\gamma) + \#E$, and the tautological ring of $\oMgn$ is spanned by classes of this form.

\subsection{A preliminary lemma}
As a starting point toward the proof of Theorem~\ref{thm:mainvanishing}, and as an illustration of our methods, we prove:

\begin{lemma} \label{le:psig+1}
Any degree-$k$ monomial in the $\psi$-classes vanishes on $\Mgn$ if $k\geq g+1$ and $n\geq 2g+3$.
\end{lemma}
\begin{proof} First, assume that $k=g+1$ and $n=2g+3$.

Let $\iota\colon\Mgg\to \M_{g,2g+3}^{ct}$ be the inclusion map, and let $A=(a_1,\ldots,a_{2g+3})$ be such that $\sum a_i=0$.  The restriction of the $\Th$-relations to $CH^{g+1}(\Mgg)$ give
\begin{equation*}
  \iota^*[s_A^*\Theta]^{g+1}= \left[\frac{1}{2}\sum_{i=1}^{2g+3}a_i^2\psi_i\right]^{g+1}=0\in CH^{g+1}(\Mgg),
\end{equation*}
where here and for the rest of the proof $\psi_i$ denotes the
$\psi$-class on $\M_{g,2g+3}$.
Eliminating $a_{2g+3}=-(a_1+\ldots+a_{2g+2})$ and dropping the $1/2$, we obtain
\begin{equation}\label{eq:psi2g+3open}
 \left[\sum_{i=1}^{2g+2}a_i^2\psi_i+(a_1+\ldots+a_{2g+2})^2\psi_{2g+3}\right]^{g+1}=0
\end{equation}
for any $(a_1\ldots,a_{2g+2})\in \ZZ^{2g+2}$. This relation is a homogeneous polynomial in the integer variables $a_i$ and the classes $\psi_i$, of degrees $2g+2$ and $g+1$, respectively, and we prove the lemma by alternatively considering \eqref{eq:psi2g+3open} as a polynomial in one set of variables or the other.

First, we view \eqref{eq:psi2g+3open} as a polynomial in the $a$-variables taking values in $CH^{g+1}(\Mgg)$. It vanishes for all integer values of the $a_i$ only if it is the zero polynomial--- in other words, only if the coefficient in front of each monomial in the $a_i$ is zero. Thus, any monomial in the $a_i$ of degree $2g+2$ gives a relation in $CH^{g+1}(\Mgg)$, which itself is a polynomial of degree $g+1$ in the $\psi$-classes. We need to check that there are enough such relations to ensure that every monomial in the $\psi$-classes of degree $g+1$ vanishes.

We now make a key observation. Consider \eqref{eq:psi2g+3open} as a polynomial in the $\psi$-classes whose coefficients are polynomials in the $a$-variables. For every $1\leq j\leq 2g+2$, $\psi_j$ appears in \eqref{eq:psi2g+3open} in the term $a_j^2\psi_j$, hence the coefficient of any $\psi$-monomial that is a multiple of $\psi_j^k$ is an $a$-polynomial that is a multiple of $a_j^{2k}$.

If we now again view \eqref{eq:psi2g+3open} as a polynomial in the $a_i$, then the observation of the preceding paragraph implies that, for every $1\leq j\leq 2g+2$, the coefficient of any monomial that is not a multiple of $a_j^{2k}$ is a polynomial in the $\psi$-classes that contains no monomials that are multiples of $\psi_j^k$. For example, the coefficient of the monomial $a_1\cdots a_{2g+2}$ does not contain the classes $\psi_1,\ldots,\psi_{2g+2}$, and indeed, it is equal to $(2g+2)!\psi_{2g+3}^{g+1}$. We thus obtain our first vanishing,
\begin{equation}
 \psi_{2g+3}^{g+1}=0\in CH^{g+1}(\Mgg)
\end{equation}
which serves as the base of a descending induction on the power of $\psi_{2g+3}$ from which we prove the claim.

Let $1\leq k\leq g+1$, and suppose that we have shown that any monomial of degree $g+1$ in the $\psi$-classes that is a multiple of $\psi_{2g+3}^{k}$ vanishes. Let $K=(k_1,\ldots,k_{2g+2})$ be non-negative integers satisfying
\begin{equation}
  \sum_{i=1}^{2g+2}k_i=g+1-(k-1),
  \label{eq:dimcount}
\end{equation}
and denote
\begin{equation*}
  \Psi_K:=\psi_1^{k_1}\cdots\psi_{2g+2}^{k_{2g+2}}.
\end{equation*}
Let $j$ be the $2(k-1)$-st element, in increasing numerical order, in the set
\[I_K:=\{i\; | \;k_i=0\}\subset\{1,\ldots,2g+2\}.\]
(The fact that $\# I_K\ge 2(k-1)$ is clear: indeed, otherwise we would have $g+1-(k-1)=\sum_{i=1}^{2g+2}k_i\ge (2g+2)-2(k-1)+1$, which would imply $k\ge g+3$.)
For $1 \leq i \leq 2g+2$, define $m_i$ by
\begin{equation*}
  m_i=
  \begin{cases}
    2k_i, & i\notin I_K, \\
    1, & i\in I_K\mbox{ and }i\leq j, \\
    0, & i\in I_K\mbox{ and }i>j.
  \end{cases}
\end{equation*}
Since $k_1+\ldots+k_{2g+2}=g+1-(k-1)$, we have $m_1+\ldots+m_{2g+2}=2g+2$. To show that $\Psi_K\psi_{2g+3}^{k-1}=0$, we consider the monomial $a_1^{m_1}\cdots a_{2g+2}^{m_{2g+2}}$.

The coefficient of this monomial in \eqref{eq:psi2g+3open} is a polynomial of degree $g+1$ in the $\psi$-classes, and by the above observation, this polynomial contains no monomial that is a multiple of $\psi_j^{k_j+1}$ for any $1\leq j\leq 2g+2$. But by \eqref{eq:dimcount}, any $\psi$-monomial of degree $g+1$ that is not a multiple of any $\psi_j^{k_j+1}$ is a multiple of either $\Psi_K$ or $\psi_{2g+3}^k$. Hence the monomial $a_1^{m_1}\cdots a_{2g+2}^{m_{2g+2}}$ imposes the following relation:
\begin{equation*}
C_K\Psi_K\psi_{2g+3}^{k-1}+[\mbox{a multiple of }\psi_{2g+3}^k]=0 \in CH^{g+1}(\Mgg),
\end{equation*}
where $C_K$ is a positive multinomial coefficient depending on $K$. By induction, all $\psi$-monomials that are multiples of $\psi_{2g+3}^k$ vanish, hence so does $\Psi_K\psi_{2g+3}^{k-1}$. This proves that any $\psi$-monomial of degree $g+1$ vanishes on $\Mgg$.

To finish the proof of the lemma, we note that, for $n>2g+3$, any $\psi$-monomial of degree $g+1$ on $\Mgn$ can be obtained using \eqref{eq:pullback} by pulling back a $\psi$-monomial from $\Mgg$ along a forgetful map, and that the vanishing of all $\psi$-monomials in degree $g+1$ trivially implies vanishing in higher degrees.
\end{proof}

\begin{remark}
\label{rem:formulas}
We have shown that the $\Theta$-relations imply the vanishing of the monomials of degree $g+1$ in the $\psi$-classes on $\Mgg$, by what is really just a multidimensional Gaussian elimination. The same Gaussian elimination can be used to obtain boundary formulas for these classes on all of $\oM_{g,2g+3}$. For this, we use Pixton's double ramification cycle relations, as stated in equation~\eqref{eq:thetaRHS}. The left-hand side of~\eqref{eq:thetaRHS} is a polynomial of degree $2g+2$ in the $a_i$, each coefficient of which is a polynomial of degree $g+1$ in the $\psi$-classes and the boundary divisors. The right-hand side is some tautological class supported on the divisor $\Delta_{irr}$ of curves with a non-separating node. Moving all boundary terms from the left to the right, we obtain an expression of the left-hand side of \eqref{eq:psi2g+3open} as a boundary class on $\oM_{g,2g+3}$.

To obtain boundary formulas for degree $g+1$ monomials in the $\psi$-classes, we proceed from equation \eqref{eq:psi2g+3open} as in the proof of Lemma \ref{le:psig+1}, but we now keep track of the boundary on the right-hand side. If we assume that Pixton's class $[\Omega_{g,A}]_{g+1}$ is a polynomial in the $a_i$, then so is $D_{g,A}$, and the proof is identical: each monomial of degree $2g+2$ in the $a_i$ imposes a relation in $\oM_{g,2g+3}$, and these relations imply the vanishing of all $\psi$-monomials on $\Mgg$, so the same Gaussian elimination produces boundary formulas for all such monomials on $\oM_{g,2g+3}$. Even without assuming polynomiality, one can construct a collection of finite difference operators\footnote{For example, if $f(x,y)=ax^2+bxy+cy^2$, then $b=f(x+1,y+1)-f(x+1,y)-f(x,y+1)+f(x,y)$.} in the $a_i$ that isolate the coefficients of any degree $2g+2$ polynomial in the $a_i$.  We then apply these operators (which are defined for any function of the $a_i$) to equation \eqref{eq:thetaRHS} and proceed as in the proof of Lemma \ref{le:psig+1}.
\end{remark}

\subsection{Proof of Theorem~\ref{thm:Fabervanishing} from the $\Theta$-relations}
In this subsection, we prove Theorem~\ref{thm:mainvanishing}; that is, we deduce Theorem~\ref{thm:Fabervanishing} from the $\Theta$-relations. This also serves as the key technical step in our proof of Theorem~$\star$.

\begin{proof}[{Proof of Theorem~\ref{thm:mainvanishing}}]
Recall that Theorem~\ref{thm:Fabervanishing} states that $R^k(\M_{g,n})$ vanishes in degrees $k \geq g$ if $n > 0$ and in degrees $k \geq g-1$ if $n=0$.  The proof of this statement is an elaboration of the technique of Lemma~\ref{le:psig+1}, which constitutes the special case in which we restrict attention to tautological classes involving only $\psi$-classes, assume that $n \geq 2g+3$, and, most importantly, prove vanishing starting in degree $g+1$ rather than $g$.  In order to obtain relations in degree $g$, we consider the $\Theta$-relations on $\M_{g,2g+3}$, multiply by appropriate classes, and push forward under a forgetful map to $\M_{g,n}$.  We must ensure, however, that the forgetful map is proper, and to this end, we must pass to the moduli space of curves with rational tails; this necessitates keeping track of those boundary terms in~\eqref{eq:Theta} that map onto the open part of the moduli space under the forgetful map.

Assume, first, that $n \leq g$; in particular, this implies that $g>0$.  We work first on the moduli space $\M_{g,2g+3}^{ct}$, so for the time being, $n$ is merely an auxiliary label.  Choose non-negative integers $c_1, \ldots, c_{2g+3}$ such that $c_i \geq 1$ for each $n+1 \leq i \leq 2g+2$, and denote
\[\Psi_C := \psi_1^{c_1} \cdots \psi_{2g+3}^{c_{2g+3}}.\]
Let $c := \sum_{i=1}^{2g+3} c_i$.

We multiply the $\Theta$-relation \eqref{eq:Theta} by $\Psi_C$ and pull back under the inclusion $\iota_{rt}\colon \Mggrt\to \M_{g,2g+3}^{ct}$.  After substituting  $a_{2g+3}=-a_1-\ldots -a_{2g+2}$ as before, we obtain:
\begin{multline}
  \label{eq:psi2g+3}
  \iota_{rt}^*[2s_A^*\Th]^{g+1}\Psi_C=\left[\sum_{i=1}^{2g+2}a_i^2\psi_i+
  \left(\sum_{i=1}^{2g+2}a_i\right)^2\psi_{2g+3}+
  \sum_{I\subset[2g+2],\# I\geq 2}a_I^2\de_0^I\right. \\
  \left.+\sum_{I\subset[2g+2],\# I\geq 1}a_{I^c}^2\de_0^{I\cup\{2g+3\}}\right]^{g+1}\Psi_C=0\in CH^{c+g+1}(\Mggrt).
\end{multline}
Here, we denote $a_{I^c}:=\sum_{i\in [2g+2]\setminus I}a_i$, and we have explicitly separated out the boundary divisors parameterizing curves with the last marked point on the rational component.  As in Lemma \ref{le:psig+1}, this relation is a homogeneous polynomial of degree $2g+2$ in the variables $a_i$, and every monomial in the $a_i$ gives a separate relation in the tautological ring, which is a polynomial of degree $c+g+1$ in the $\psi$-classes and boundary divisors on $\Mggrt$.

We now make the following important observation.

\begin{claim}\label{claim1} The coefficient in~\eqref{eq:psi2g+3} of any $a$-monomial  that is a multiple of $a_1\cdots a_n$ is the sum of a polynomial in only the $\psi$-classes and a collection of terms supported on the boundary, and each boundary term that occurs is either a multiple of a divisor $\delta_0^I$ having $\#(I\cap[n])\geq 2$ or of a divisor $\de_0^{I\cup\{2g+3\}}$ having $\#(I\cap[n])\geq 2$.
\end{claim}

\begin{proof}[{Proof of Claim~\ref{claim1}.}] First, we note that, according to Lemma \ref{le:vanish}, for any $I\subset[2g+2]$ we have
\begin{equation}
  \de_0^I\Psi_C=0\mbox{ if }\#(I\cap [n])\leq 1,
  \label{eq:bvanish1}
\end{equation}
\begin{equation}
  \de_0^{I\cup\{2g+3\}}\Psi_C=0\mbox{ if }I\cap [n]=\emptyset.
  \label{eq:bvanish2}
\end{equation}
Therefore, the only boundary divisors that appear in \eqref{eq:psi2g+3} are $\de_0^I$ with $\#(I\cap [n])\geq 2$ and $\de_0^{I\cup\{2g+3\}}$ with $\#(I\cap [n])\geq 1$.

Now, let $I\subset[2g+2]$ be such that $I\cap [n]=\{i\}$.  Suppose that $\de_0^{I\cup\{2g+3\}}$ appears in~\eqref{eq:psi2g+3} in an $a$-monomial that is a multiple of $a_i$.  Then, since $a_{I^c}^2$ does not involve $a_i$, this can only occur if $\de_0^{I \cup \{2g+3\}}$ is multiplied by one of the classes $\psi_i$, $\psi_{2g+3}$, $\delta_0^J$ for $i\in J$, or $\delta_0^{J\cup \{2g+3\}}$ for $i\notin J$.  In the first two cases, the product is zero by Lemma~\ref{le:vanish}:
\begin{equation*}
  \de_0^{I\cup \{2g+3\}}\psi_i\Psi_C=\de_0^{I\cup \{2g+3\}}\psi_{2g+3}\Psi_C=0.
\end{equation*}
In the third case, the product is zero whenever $\#(J \cap [n])=1$, by the previous paragraph.  The product is also zero in the fourth case whenever $\# (J \cap [n])=1$, because
\begin{equation*}
  \de_0^{I\cup \{2g+3\}} \de_0^{J\cup \{2g+3\}}=0
\end{equation*}
represents a geometrically empty intersection.

Now, consider the coefficient of any $a$-monomial in \eqref{eq:psi2g+3} that is a multiple of $a_i$. This coefficient consists of a polynomial only in the $\psi$-classes, as well as an expression supported on the boundary. The previous paragraph shows that any term in the boundary part containing $\de_0^{I \cup \{2g+2\}}$ with $I\cap [n]=\{i\}$ also contains at least one boundary divisor $\de_0^J$ or $\de_0^{J \cup \{2g+3\}}$ with $\#(J \cap [n]) \geq 2$.  Hence, given an $a$-monomial in \eqref{eq:psi2g+3} that is a multiple of $a_1\cdots a_n$, applying the above reasoning for each $a_i$ proves the claim.
\end{proof}

The importance of Claim~\ref{claim1} is the following. Let $\Pi_n:=\pi_{n+1}\circ\cdots\circ\pi_{2g+3}$ denote the map forgetting the points $p_{n+1},\ldots,p_{2g+3}$:
\begin{equation*}
  \Pi_n\colon\Mggrt\to\Mgnrt.
\end{equation*}
If $(C,p_1,\ldots,p_{2g+3})$ lies on a boundary divisor $\De_0^I$ or $\De_0^{I\cup\{2g+3\}}$ having $\#(I\cap[n])\geq 2$, then it remains singular after forgetting the points $p_{n+1},\ldots,p_{2g+3}$ and stabilizing. In other words, the pushforward along $\Pi_n$ of any tautological class on $\Mggrt$ that is a multiple of any of these divisors is a boundary stratum on $\Mgnrt$. Hence, if we take the pushforward of \eqref{eq:psi2g+3} to $\Mgnrt$, and only consider relations that come from $a$-monomials that are multiples of $a_1\cdots a_n$, any terms involving boundary divisors on $\Mggrt$ vanish when restricted to $\Mgn$.

We can get rid of all $a$-monomials that are not multiples of $a_1 \cdots a_n$ by formally differentiating \eqref{eq:psi2g+3} with respect to these variables.  Thus, if $\iota_n\colon \M_{g,n} \rightarrow \Mgnrt$ denotes the inclusion, we have shown the following relation in $CH^{c-g-2+n}(\M_{g,n})$:
\begin{equation}
\frac{\partial^n}{\partial a_1\cdots\partial a_n}\iota_n^*\Pi_{n*}\left[\sum_{i=1}^{2g+2}a_i^2\psi_i+
  \left(\sum_{i=1}^{2g+2}a_i\right)^2\psi_{2g+3}\right]^{g+1}\Psi_C=0
\label{eq:psidiff}
\end{equation}
whenever $c_i \geq 1$ for $n+1 \leq i \leq 2g+2$.

This implies the vanishing of a collection of tautological classes on $\Mgn$:

\begin{claim}\label{claim2} Assume that $n\leq g$, and let $d_1,\ldots,d_{2g+3}$ be non-negative integers satisfying
\begin{itemize}
\item $\displaystyle d:=\sum_{i=1}^{2g+3} d_i\geq 3g+3-n,$
\item $\displaystyle d_i\geq 1\mbox{ for }i\geq n+1,$
\item $\displaystyle \#\,\big\{i\;\big|\;d_i=0\big\}\leq d_{2g+3}.$
\end{itemize}
Then
\begin{equation}\label{eq:pushpsi}
\iota_n^*\Pi_{n*}[\psi_1^{d_1}\cdots\psi_{2g+3}^{d_{2g+3}}]=0\in CH^{d+n-2g-3}(\Mgn).
\end{equation}

\end{claim}

\begin{proof}[{Proof of Claim~\ref{claim2}.}]
We first rewrite the $\psi$-monomial in question in notation similar to what appears in \eqref{eq:psi2g+3}.  Namely, we claim that we can write
\begin{equation*}
\psi_1^{d_1}\cdots\psi_{2g+3}^{d_{2g+3}}=\psi_1^{k_1}\cdots\psi_{2g+3}^{k_{2g+3}}\Psi_C,
\end{equation*}
where $\Psi_C$ is a multiple of $\psi_{n+1}\cdots\psi_{2g+2}$, and the integers $k_i$ satisfy
\begin{equation}\label{eq:claimcond}
\sum_{i=1}^{2g+3} k_i=g+1, \quad \min(1,\#\{i\;|\;1\leq i\leq n,\,k_i=0\})\leq k_{2g+3}.
\end{equation}
Indeed, set
\begin{equation*}
  k'_i=
  \begin{cases}
    1, & i\leq n\mbox{ and }d_i>0, \\
    0, & d_i=0, \\
    \min\left(1,\#\left\lbrace i\;\big|\;d_i=0\right\rbrace\right), & i=2g+3.
  \end{cases}
\end{equation*}
Then the integers $k_i'$ satisfy the conditions
\begin{equation*}
k_i'\leq d_i\mbox{ for }1\leq i\leq n,\quad k_i'\leq d_i-1\mbox{ for }n+1\leq i\leq 2g+2,
\end{equation*}
\begin{equation*}
 \min\left(1,\#\left\lbrace i\;|\;1\leq i\leq n,\,k'_i=0\right\rbrace\right)\leq k'_{2g+3}.
\end{equation*}
Now let $k_i$ be any collection of integers that satisfy the same inequalities as above, such that $k_i'\leq k_i$, and such that they add to $g+1$.

To prove the claim, it is therefore enough to show that
\begin{equation}\label{eq:pushpsi2}
\iota^*_n\Pi_{n*}[\psi_1^{k_1}\cdots\psi_{2g+3}^{k_{2g+3}}\Psi_C]=0 \in CH^{c+n-g-2}(\Mgn)
\end{equation}
for any collection of integers $k_1,\ldots,k_{2g+3}$ satisfying \eqref{eq:claimcond} and any $\psi$-monomial $\Psi_C$ of degree $c\geq 2g+2-n$ that is a multiple of $\psi_{n+1}\cdots\psi_{2g+2}$.

We prove \eqref{eq:pushpsi2} from the relations \eqref{eq:psidiff} by an argument essentially identical to the proof of Lemma \ref{le:psig+1}, via descending induction on $k_{2g+3}$.  First, we prove \eqref{eq:pushpsi2} for the case $(k_1, \ldots, k_{2g+3}) = (0, \ldots, 0, g+1)$, which satisfies \eqref{eq:claimcond} because $n \leq g$.  In this case we have
\begin{equation*}
\iota_n^*\Pi_{n*}[\psi_{2g+3}^{g+1}\Psi_C]=0,
\end{equation*}
since $(2g+2)!\iota_n^*\Pi_{n*}[\psi_{2g+3}^{g+1}\Psi_C]$ is the coefficient of $a_{n+1}\cdots a_{2g+2}$ in \eqref{eq:psidiff}.

Now let $1\leq k\leq g+1$, and assume that we have shown that equation \eqref{eq:pushpsi2} holds for any $(k_1,\ldots,k_{2g+3})$ satisfying condition~\eqref{eq:claimcond} and such that $k_{2g+3}\geq k$. Let $K=(k_1,\ldots,k_{2g+2})$ be a collection of integers such that $(k_1,\ldots,k_{2g+2},k-1)$ satisfies  \eqref{eq:claimcond}, and denote
\begin{equation*}
\Psi_K:=\psi_1^{k_1}\cdots\psi_{2g+2}^{k_{2g+2}}.
\end{equation*}
We need to show that $\iota^*_n \Pi_{n*}[\Psi_K \psi_{2g+3}^{k-1}\Psi_C]=0$.

Let $j$ be the $2(k-1)$-st element, in increasing numerical order, in the set
\[I_K:=\{i\; | \;k_i=0\}\subset\{1,\ldots,2g+2\}.\]
Note that by assumption \eqref{eq:claimcond}, we have $j\geq n$.  For $1 \leq i \leq 2g+2$, define $m_i$ by
\begin{equation*}
  m_i:=
  \begin{cases}
    2k_i, & i\notin I_K, \\
    1, & i\in I_K\mbox{ and }i\leq j, \\
    0, & i\in I_K\mbox{ and }i>j.
  \end{cases}
\end{equation*}
Since $k_1+\ldots+k_{2g+2}=g+1-(k-1)$, we see that $m_1+\ldots+m_{2g+2}=2g+2$. Furthermore, $m_i>0$ if $i\leq n$. Therefore, the monomial $a_1^{m_1-1}\cdots a_n^{m_n-1}a_{n+1}^{m_{n+1}}\cdots a_{2g+2}^{m_{2g+2}}$ occurs in equation \eqref{eq:psidiff}. A careful inspection shows that its coefficient is equal to
\begin{equation*}
C_K\iota^*_n \Pi_{n*}[\Psi_K \psi_{2g+3}^{k-1}\Psi_C]+\iota^*_n \Pi_{n*}[\mbox{a multiple of }\psi_{2g+3}^k\Psi_C]=0,
\end{equation*}
where $C_K$ is a nonzero coefficient, and each $\psi$-monomial that occurs in the second summand satisfies condition \eqref{eq:claimcond}. By induction, all these summands vanish, hence $\iota^*_n \Pi_{n*}[\Psi_K \psi_{2g+3}^{k-1}\Psi_C]=0$.
\end{proof}

In fact, relations \eqref{eq:pushpsi} imply the vanishing of all tautological classes of degree $g$ on $\Mgn$. To see this, recall that the tautological ring of $\M_g$ is generated by the $\ka$-classes, while the tautological ring of $\Mgn$ is generated by the $\ka$- and the $\psi$-classes.  We define the {\it length} of a monomial $\ka_{b_1}\cdots\ka_{b_l}$ in the $\ka$-classes to be $l$.

\begin{claim} Let $g \geq 1$.  Then any codimension-$k$ monomial in the $\ka$-classes vanishes on $\M_{g,1}$ for any $k\geq g$.\label{claim3}
\end{claim}
\begin{proof}[{Proof of Claim~\ref{claim3}}]
The claim is proved by induction on the length.  The only length one monomials are $\ka_k$, and using \eqref{eq:pushpsi} and \eqref{eq:pushka}, we see that
\begin{equation*}
  \iota_1^*\Pi_{1*}[\psi_2^{k+1}\psi_3\cdots\psi_{2g+2}\psi_{2g+3}]=C\cdot \ka_{k}=0\mbox{ on }\M_{g,1},
\end{equation*}
where $C$ is a nonzero constant. Hence, $\ka_k$ vanishes on $\M_{g,1}$.

Now, suppose that the claim has been proven for all $\ka$-monomials of length less than $l$.  Let $b_1,\ldots,b_{l}$ be positive integers. If $l>2g+2$, then the monomial $\ka_{b_1}\cdots \ka_{b_{l-1}}$ has degree at least $2g+2$ and length $l-1$, and thus vanishes by the induction assumption, hence $\ka_{b_1}\cdots \ka_{b_l}$ vanishes as well.
If $l\leq 2g+2$, then using \eqref{eq:pushpsi}, \eqref{eq:pushka}, and \eqref{eq:observation}, we see that
\[\iota_1^*\Pi_{1*}[\psi_2^{b_1+1}\cdots\psi_{l+1}^{b_{l}+1}\psi_{l+2}\cdots\psi_{2g+3}]=C\cdot\ka_{b_1}\cdots \ka_{b_l}  +(\ka\mbox{-monomials of length less than } l)=0\mbox{ on }\M_{g,1},\]
where $C$ is a nonzero constant. Hence, $\ka_{b_1}\cdots \ka_{b_l}$ vanishes on $\M_{g,1}$ by induction.
\end{proof}

\begin{claim}
Any codimension-$k$ monomial in the $\psi$- and $\ka$-classes that is a multiple of $\psi_1\cdots\psi_n$ vanishes on $\Mgn$ for any $k\geq g$.
\label{claim4}
\end{claim}

\begin{proof}[{Proof of Claim~\ref{claim4}}] First, assume that $n\leq g$.  The monomials in question have the form
\[\psi_1^{d_1} \cdots \psi_n^{d_n} \kappa_{b_1}\cdots \kappa_{b_l},\]
where $b_i$ and $d_i$ are positive integers. We denote $m:=b_1 + \ldots + b_l$ so that  $d_1 + \ldots + d_n = k-m$.  With $m$ fixed, we proceed by induction on the length $l$ of the $\kappa$-monomial.

The only degree-$m$ monomial in the $\kappa$ classes of length one is $\ka_m$, and using \eqref{eq:pushpsi} and \eqref{eq:pushka}, we see that
\begin{equation*}
 \iota_n^* \Pi_{n*}[\psi_1^{d_1}\cdots\psi_n^{d_n}\psi_{n+1}\cdots\psi_{2g+2}\psi_{2g+3}^{m+1}]=C\cdot \psi_1^{d_1}\cdots\psi_n^{d_n}\ka_m=0\mbox{ on }\Mgn,
\end{equation*}
where $C$ is a nonzero constant, so $\psi_1^{d_1}\cdots\psi_n^{d_n}\ka_m$ vanishes on $\Mgn$. The above formula also holds when $m=0$ and $\ka_0=2g-2+n$, showing that every monomial in only the $\psi$-classes vanishes.

Now suppose that we have shown the claim for every monomial of codimension $m$ and length less than $l$. Then
\begin{multline*}
 \iota_n^* \Pi_{n*}[\psi_1^{d_1}\cdots\psi_n^{d_n}\psi_{n+1}^{b_1+1}\cdots\psi_{n+l}^{b_l+1}\psi_{n+l+1}\cdots\psi_{2g+3}] \\
  =C\psi_1^{d_1}\cdots\psi_n^{d_n}\cdot\left[\ka_{b_1}\cdots \ka_{b_l}+(\ka\mbox{-monomials of length less than }l)\right]=0
\end{multline*}
on $\Mgn$, where $C$ is a nonzero constant. Hence, $\psi_1^{d_1}\cdots\psi_n^{d_n}\ka_{b_1}\cdots \ka_{b_l}$ vanishes on $\Mgn$ by induction.

Finally, let $n>g$.  If $g=0$, then the claim is clearly true for dimension reasons.  For $n> g$, the above shows that $\psi_1\cdots\psi_g$ vanishes on $\M_{g,g}$, so by \eqref{eq:pullback} it vanishes on $\Mgn$.  It follows that $\psi_1\cdots\psi_n$ and any multiple of it vanishes as well.
\end{proof}

We are now ready to prove the main theorem: that any degree-$k$ monomial in the $\psi$- and $\kappa$-classes vanishes on $\M_{g,n}$ for $k \geq g$, and additionally, for $k = g-1$ and $n=0$.

For $n=1$, this is immediate from the above: if the monomial contains $\psi_1$, then it vanishes by Claim~\ref{claim4}, and if not, it vanishes by Claim~\ref{claim3}.

To prove the claim for $n=0$, note that the forgetful map $\pi\colon\M_{g,1}\to \Mg$ is proper. Since $\psi_1^{k}$ vanishes on $\M_{g,1}$ for $k>g$, we have
\begin{equation*}
\ka_k=\pi_*\psi_1^{k+1}=0\mbox{ on }\Mg\mbox{ for }k\geq g-1.
\end{equation*}
More generally, suppose that we have shown that any $\ka$-monomial of degree $k\geq g-1$ and length less than $l$ vanishes on $\Mg$. Let $b_1,\ldots,b_l$ be positive integers such that $b_1+\ldots+b_l=k\geq g-1$. By equation \eqref{eq:pullback}, we have
\begin{equation*}
\ka_{b_1}\cdots\ka_{b_l}\psi_1=(\pi^*\ka_{b_1}+\psi_1^{b_1})\cdots(\pi^*\ka_{b_l}+\psi_1^{b_l})\psi_1
\end{equation*}
on $\M_{g,1}$. The left-hand side has degree at least $g$, hence it vanishes on $\M_{g,1}$ by Claim~4.  Pushing forward the right-hand side under $\pi$, we get a multiple of $\ka_{b_1}\cdots\ka_{b_l}$ and a sum of $\ka$-monomials of lengths less than $l$. Hence, $\ka_{b_1}\cdots\ka_{b_l}$ vanishes on $\Mg$, proving the claim for $n=0$.

These cases provide the base for an induction on $n$, assuming $g>0$; for $g=0$, the obvious vanishing of the classes $\psi_i$ and $\kappa_i$ on $\M_{0,3}$ can be used as the base.  Now, suppose that we have proven the claim for $\Mgn$, and consider a monomial $\Xi=\ka_{b_1}\cdots\ka_{b_l}\psi_1^{d_1}\cdots\psi_{n+1}^{d_{n+1}}$ on $\M_{g,n+1}$ of degree
\begin{equation*}
b_1+\ldots+b_l+d_1+\ldots+d_{n+1}=k\geq g.
\end{equation*}
Proceed by descending induction on the number $D$ of positive $d_i$. If $D=n+1$, then all of the $d_i$ are positive, and $\Xi$ vanishes by Claim~\ref{claim4}. If not, then without loss of generality we can assume that $d_{n+1}=0$. Using equation \eqref{eq:pullback}, we see that
\begin{equation*}
\pi_{n+1}^*[\ka_{b_1}\cdots\ka_{b_l}\psi_1^{d_1}\cdots\psi_{n}^{d_{n}}]=
(\ka_{b_1}-\psi_{n+1}^{b_1})\cdots(\ka_{b_l}-\psi_{n+1}^{b_l})\psi_1^{d_1}\cdots\psi_{n}^{d_{n}}.
\end{equation*}
The left-hand side is the pullback of a class from $\Mgn$, which is zero by induction on $n$, and the expansion of the right-hand side is the sum of $\Xi$ and a collection of classes that vanish by induction on $D$. Hence, $\Xi$ vanishes on $\M_{g,n+1}$. This completes the proof of the theorem.
\end{proof}

\begin{remark}
\label{rem:formulas2}
We have shown that the $\Theta$-relations imply the vanishing of the tautological ring of $\Mgn$ in degrees $g-\delta_{0n}$ and above.  Just as in Remark~\ref{rem:formulas}, however, by upgrading the $\Theta$-relations to Pixton's double ramification cycle relations in the form \eqref{eq:thetaRHS}, and by keeping track of the boundary terms throughout the above computations, the same proof produces expressions for these classes as tautological classes supported on the boundary.
\end{remark}

Thus, we have the following corollary of the proof of Theorem~\ref{thm:mainvanishing}.

\begin{corollary}
Any polynomial of degree $k$ in the $\ka$- and $\psi$-classes on $\oMgn$ is equivalent to an algorithmically computable tautological class supported on the boundary of $\oMgn$, where $k\geq g$ or $n=0$ and $k=g-1$.  \label{cor:boundary}
\end{corollary}

\subsection{Proof of Theorem~$\star$ from Pixton's double ramification cycle relations}\label{proofstar}
We are now ready to give a constructive proof of Theorem~$\star$, thus completing the proof of Theorem~\ref{thm:mainthm}.

\begin{proof}[{Proof of Theorem~\ref{thm:mainthm}}]
Let $\Ga$ be a stable graph of genus $g$ with $n$ legs.  Let $\ga$ be a basic class on $\oM_{\Ga}$ of the form~\eqref{eq:gamma}, and let $\xi_{\Ga*}(\ga)$ be the corresponding tautological class on $\oMgn$, where $\xi_{\Ga}$ is the gluing map~\eqref{eq:xi}.

We apply Corollary~\ref{cor:boundary} to every $\ga_v$. If $d(\ga_v)\geq g(v)$, then we can express $\ga_v$ as a boundary class using Pixton's double ramification cycle relations and the algorithm of Theorem~\ref{thm:mainvanishing}.  (If $g(v)=0$ or $g(v)=1$, we can alternatively use the divisorial relations~\eqref{eq:M11} and \eqref{eq:TRR}.)

Therefore, possibly after replacing $\Gamma$ by a graph representing a deeper boundary stratum, we can assume that $d(\gamma_v) \leq g(v)-1$ if $g(v)>0$ and that $d(\gamma_v) = 0$ if $g(v)=0$.  Let $v_0$ denote the number of genus-zero vertices of $\Gamma$. Using~\eqref{eq:genus}, we see that
\begin{align*}
\deg \xi_{\Ga*}(\ga)&=\sum_{v\in V}d(\ga_v)+\# E\leq \sum_{v\in V,g(v)>0}(g(v)-1)+\# E\\
&=g(\Ga)-h^1(\Ga_0)+v_0-\# V+\# E=g(\Ga)+v_0-1.
\end{align*}
Hence, $v_0\geq \deg \xi_{\Ga*}(\ga)-g(\Ga)+1$, proving the theorem.
\end{proof}

\begin{remark} The proofs of Theorem~\ref{thm:mainvanishing} and Theorem~\ref{thm:mainthm} together provide an explicit algorithm for expressing any tautological class of codimension $k\geq g$ on $\oMgn$ as a boundary class having $k-g+1$ rational components. We provide an outline for this algorithm below to demonstrate its feasibility; we (or our students) plan to implement this algorithm and make it publicly available in the future.

In the notation of Section~\ref{sec:strata}, let $(\Ga,\ga)$ be a marked stable graph defining a tautological class $\xi_{\Ga*}(\ga)\in R^k(\oMgn)$ with a basic class $\ga_v$ at each vertex $v\in V$. We say that $\xi_{\Ga*}(\ga)$ has {\em property $\star$} if
\begin{equation*}
\deg (\ga_v) \leq \max (g(v)-1,0)\mbox{ for every }v\in V.
\end{equation*}
By the proof of Theorem~\ref{thm:mainthm}, this implies that $\Ga$ has at least $k-g+1$ rational components. Our goal is to obtain an expression for every tautological class in $R^*(\oMgn)$ as a linear combination of classes having property $\star$. We obtain such expressions by induction on the genus and the degree. Assume that we have already constructed such formulas for all classes in $R^*(\oM_{g',n})$ with $g'<g$ and all $n$, and for all classes in $R^{k'}(\oMgn)$ with $k'<k$ and all $n$.

For each $v\in V$ such that $g(v)<g$ and $\deg(\ga_v)\geq g(v)$, we use our database to express $\ga_v$ as a linear combination of classes having property $\star$. In addition, there may be at most one vertex $u$ such that $g(u)=g$ and $\deg (\ga_u)\geq g$. In this case, we use the algorithm of Theorem~\ref{thm:mainvanishing} to express $\ga_u$ as a linear combination of non-trivial boundary classes. The highest possible degree of a basic class on any such boundary class is $k-1$, and by induction, our database expresses all such classes as linear combinations of classes having property $\star$. Hence, we obtain such an expression for $\ga_u$ as well. Gluing together these formulas, we obtain an expression for $\xi_{\Ga*}(\ga)$ in terms of classes having property $\star$. In this way, we obtain an expression for any given tautological class in terms of classes having property $\star$ in a finite number of steps.
\end{remark}

\section{Example}
\label{examples}
In this section, we exemplify our methods by reproving the divisorial formulas~\eqref{eq:M11} expressing $\psi_1$ and $\ka_1$ in terms of the boundary divisor $\delta_{irr}$ on $\M_{1,1}$; note that this is not a circular argument, as these formulas were not used in the derivation of the main theorem.

Before we begin, we note that the genus-zero divisorial formulas~\eqref{eq:TRR} follow from the pullback formulas~\eqref{eq:pullback} and from the relation $\psi_1=0$ on $\oM_{0,3}$, which can be formally obtained from relation~\eqref{eq:Thetag+1} by substituting $a_3=-a_1-a_2$ and taking the coefficient of $a_1a_2$.

\subsection{Pixton's class}
\label{subsec:Pixton}

We first recall the definition of Pixton's class $\Omega_{g,A}$.

Define auxiliary classes $\Omega_{g, A}^r$ depending on
an additional integer parameter $r > 0$ as follows. Let $\Gamma=(V,H,g,p,\iota)$ be a stable graph of genus $g$ with $n$ legs (following the notation of Section~\ref{subsec:technical}), and let $A=(a_1,\ldots,a_n)\in \ZZ^n$. A \emph{weighting modulo
  $r$} on $\Gamma$ is a map
\begin{equation*}
  w\colon H \to \{0, \dotsc, r - 1\}
\end{equation*}
satisfying three properties:
\begin{enumerate}
\item For any $i \in \{1, \dotsc, n\}$ corresponding to a leg $\ell_i$ of $\Gamma$, we have $w(\ell_i) \equiv a_i \pmod{r}$.
\item For any edge $e\in E$ corresponding to two half-edges $h, h'\in H$, we
  have $w(h) + w(h') \equiv 0 \pmod{r}$.
\item For any vertex $v \in V$, we have
  $\sum_{h\in p^{-1}(v)} w(h) \equiv 0 \pmod{r}$.
\end{enumerate}
(Cf. the discussion of weights in Section \ref{subsec:Thetarelations}.)  Define $\Omega_{g, A}^r$ to be the class
\begin{equation}
  \label{eq:pixton}
  \sum_{\Gamma,w} \frac{1}{\#\Aut(\Gamma)} \frac 1{r^{h^1(\Gamma)}} \xi_{\Gamma*}\left(\prod_{i=1}^n e^{\frac 12 a_i^2\psi_i} \prod_{(h,h')\in E} \frac{1 - e^{-\frac 12 w(h)w(h')(\psi_h + \psi_{h'})}}{\psi_h + \psi_{h'}}\right),
\end{equation}
where the sum is over all isomorphism classes of stable graphs
$\Gamma$ together with a weighting $w$ modulo $r$.
Pixton has proven that the class $\Omega_{g, A}^r$ is a polynomial in
$r$ for $r \gg 0$ (see \cite[Appendix]{2016JandaPandharipandePixtonZvonkine}).
The class $\Omega_{g, A}$ is then defined as the constant term of this
polynomial in $r$.

All stable graphs $\Gamma$ corresponding to curves of compact type are
trees.
When $\Gamma$ is a tree, then there exists a unique weighting modulo
$r$, and when $r > \frac 12 |\sum_{i = 1}^n a_i|$, formula
\eqref{eq:pixton} is essentially obtained by expanding the formula for
$\exp([s_A^*\Theta])$ and performing repeated intersections of
divisors on $\Mgnct$.

\subsection{Computing $\kappa_1$}

Equipped with Pixton's formula for $\Omega_{g,A}$, we proceed with the computation of $\kappa_1$ by computing the coefficient of $a_1 a_2 a_3 a_4$ in
\begin{equation}
  \label{eq:ex}
  2\Pi_{1*} (\psi_2\psi_3\psi_4 [\Omega_{1, A}]_2) = 0,
\end{equation}
where the map $\Pi_1\colon \oM_{1, 5} \to \oM_{1, 1}$ forgets all but the
first marking.

Let us first consider the simpler question of computing $\kappa_1$ on
$\Mct_{1, 1}$.  In this case, we can replace $2[\Omega_{1, A}]_2$ by
$[s_A^*\Theta]^2$.  Since $\Mrt_{1, n} = \Mct_{1, n}$, we can use
\eqref{eq:psi2g+3} to compute $[s_A^*\Theta]^2$.

We claim that the coefficient of $a_1 a_2 a_3 a_4$ in
$\psi_2\psi_3\psi_4[s_A^*\Theta]^2$ is equal to
\begin{equation}
\label{eq:a1234}
  \frac 14 \cdot 24 \psi_2\psi_3\psi_4\psi_5^2.
\end{equation}
To see this, first notice that we can remove the summands
$a_i^2\psi_i$ for $i \in \{1, 2, 3, 4\}$ from \eqref{eq:psi2g+3},
since they will not give a multiple of $a_1 a_2 a_3 a_4$.  Next,
recall that multiplying with the class $\psi_2\psi_3\psi_4$ kills all
boundary divisor classes $\delta_0^I$ for
$I \subset \{1, \dotsc, 5\}$ except when $\{1, 5\} \subset I$.
Thus the coefficient of $a_1 a_2 a_3 a_4$ in $\psi_2\psi_3\psi_4[s_A^*\Theta]^2$ equals
the coefficient of $a_1 a_2 a_3 a_4$ in
\begin{equation*}
  \frac 14 \cdot \psi_2\psi_3\psi_4\left((a_1+a_2+a_3+a_4)^2\psi_5 - \sum_{I \subset \{1, 2, 3, 4\}, 1 \in I} \left(\sum_{i \notin I} a_i\right)^2 \delta_0^{I \cup \{5\}} \right)^2.
\end{equation*}
It remains to show that only the multiple of $\psi_5^2$
contributes.  This is true since, on the one hand, $\psi_2\psi_3\psi_4\psi_5$ kills any of the boundary
divisors by Lemma \ref{le:vanish}, and on the other hand, in the square of the boundary terms
the variable $a_1$ does not appear.  Since the coefficient of
$a_1 a_2 a_3 a_4$ in $(a_1 + a_2 + a_3 + a_4)^2$ is 24, we
obtain formula \eqref{eq:a1234}.

Thus, the coefficient of $a_1 a_2 a_3 a_4$ in
$2\Pi_{1*} (\psi_2\psi_3\psi_4 [\Omega_{1, A}]_2)$ is equal to
\begin{equation}\label{eq:kade}
   \frac 14 \cdot 24^2 \kappa_1 + C\delta_{irr}
\end{equation}
for a constant $C$ that we now need to determine.

To compute $C$, we only need to consider stable graphs $\Gamma$ with
$h^1(\Gamma) = 1$.  We claim that there is only one such dual graph with non-zero
coefficient of $a_1 a_2 a_3 a_4$ .

To see why this is the case, let us first look at the stable graph
$\Gamma$ with exactly one vertex and a loop $e = (h, h')$.  There
exist $r$ weightings modulo $r$ on $\Gamma$, which can be
distinguished by the value of $w(h) \in \{0, \dotsc, r-1\}$.  Since we
only need the degree-2 part of Pixton's class for \eqref{eq:ex}, and since the edge term in
\eqref{eq:pixton} for $\Gamma$ does not depend on the $a_i$, the
summand in \eqref{eq:pixton} for $\Gamma$ is a non-homogeneous
polynomial in the $a_i$ of degree 2.  Thus, it gives a zero
coefficient of $a_1 a_2 a_3 a_4$.

By similar arguments as in the compact-type case, the only remaining
stable graphs $\Gamma$ whose contribution will not be killed by
$\psi_2\psi_3\psi_4$ have two vertices $v_1, v_2$ connected by a pair of edges $e_1, e_2$ such that the leg $\ell_5$ associated to the fifth marked point lies on $v_1$ and the leg $\ell_1$ associated to the first marked point lies on $v_2$.  Let us write $e_i = (h_i, h'_i)$, where $h_i$ is the half-edge at
vertex $v_1$ and $h'_i$ is the half-edge at $v_2$.  There are again
$r$ choices of weightings modulo $r$, indexed by $w(h_1)$.  The locus
in $\oM_{1, 5}$ corresponding to $\Gamma$ is of codimension 2, and
therefore we need to take the constant term in the factors
corresponding to the legs in \eqref{eq:pixton}.  Notice that
\begin{equation*}
  w(h_1') \equiv -w(h_1), \qquad w(h_2) \equiv x - w(h_1), \qquad w(h_2') \equiv w(h_1) - x,
\end{equation*}
where
\begin{equation*}
  x = \sum_{i\text{ at }v_2} a_i,
\end{equation*}
and therefore the contribution of $\Gamma$ to $[\Omega_{1, A}]_{2}$
depends on the $a_i$ only in the quantity $x$.  Thus, we can only have a
non-zero coefficient of $a_1 a_2 a_3 a_4$ when
$\ell_1, \ell_2, \ell_3, \ell_4$ are on $v_2$; this is the unique
stable graph $\Gamma$ that contributes.

We now compute the contribution of $\Gamma$.  As we have seen,
it depends on the $a_i$ only in the form $x = a_1 + a_2 + a_3 + a_4$.
For the computation we can assume that $x$ is positive.  Let us also
write $a := w(h_i)$.  Since the edge term corresponding to $e_1$ in
\eqref{eq:pixton} vanishes when $a = 0$, we can assume that
$a \in \{1, \dotsc, r - 1\}$.  In the case that $a \le x$, we can
write
\begin{equation*}
  w(h_1) = a, \quad w(h'_1) = r-a, \quad w(h_2) = x-a, \quad w(h'_2) = r+a-x.
\end{equation*}
Otherwise,
\begin{equation*}
  w(h_1) = a, \quad w(h'_1) = r-a, \quad w(h_2) = r+x-a, \quad w(h'_2) = a-x.
\end{equation*}
By inclusion-exclusion, the contribution of $\Gamma$ is therefore given
by
\begin{multline*}
  \frac 14 \frac 1r \Bigg(\sum_{a = 1}^{r-1} a(r-a)(r+x-a)(a-x) \\
  + \sum_{a = 1}^x a(r-a)((x-a)(r+a-x) - (r+x-a)(a-x))\Bigg).
\end{multline*}
The first sum gives only a polynomial of degree 2 in $x$ and will
therefore not lead to a coefficient of $a_1 a_2 a_3 a_4$.  From the
remaining summand we obtain
\begin{equation*}
  \frac 14 \sum_{a = 1}^x a(r-a)(2(x-a)) \equiv -\frac 14 \cdot 2\sum_{a = 1}^x a^2(x-a) \pmod{r},
\end{equation*}
which, by computing the power sum, equals
\begin{equation*}
  \frac 14\left(-\frac 16 x^4 - \frac 23 x^3 + \frac 76 x^2 - \frac 13 x\right).
\end{equation*}
The total contribution of $\Gamma$ to the coefficient of
$a_1 a_2 a_3 a_4$ in \eqref{eq:ex} is therefore
\begin{equation*}
  -2 \cdot \frac 14 \cdot \frac{24}6 \frac 1{\#\Aut(\Gamma)} \Pi_{1*}(\psi_2\psi_3\psi_4 \xi_{\Ga*}[1]) = -\frac 14 \cdot 48 \delta_{irr}.
\end{equation*}
Plugging in this value of $C$ into equation~\eqref{eq:kade}, we conclude the well-known formula
\begin{equation*}
  \kappa_1 = \frac 1{12} \delta_{irr} \in CH^*(\oM_{1, 1}).
\end{equation*}

\subsection{Computing $\psi_1$}

For the computation of $\psi_1$ in terms of $\kappa_1$ and $\delta_{irr}$,
we now consider the coefficient of $a_1^2 a_2 a_3$ in \eqref{eq:ex}.

Similarly to the computation for $\kappa_1$, only the multiples of
$\psi_5^2$ and $\psi_1\psi_5$ give coefficients that are divisible by $a_1$ and are not killed by the multiplication
with $\psi_2\psi_3\psi_4$.  Hence, the
coefficient of $a_1^2 a_2 a_3$ in
$2\Pi_{1*} (\psi_2\psi_3\psi_4 [\Omega_{1, A}]_2)$ is equal to
\begin{equation*}
   0=\frac 14 \cdot 12 \cdot 24 \kappa_1 + \frac 14 \cdot 4 \cdot 24 \psi_1 + C'\delta_{irr}
\end{equation*}
for some constant $C'$.

Similarly to the computation for $\kappa_1$, the only
dual graphs $\Gamma$ that can contribute to $C'$ must have two vertices
$v_1, v_2$ connected by a pair of edges $e_1, e_2$ with $\ell_5$ on
$v_1$ and $\ell_1$ on $v_2$.  Let us write $e_i = (h_i, h'_i)$, where
$h_i$ is the half-edge at vertex $v_1$ and $h'_i$ is the half-edge at
$v_2$.  Unlike for $\kappa_1$, not only the dual graph $\Gamma$ for which $\ell_1, \ell_2, \ell_3, \ell_4$ are all on $v_2$ contributes
to $C'$, but also the dual graph $\Gamma'$ whose set of legs at $v_2$
is $\{\ell_1, \ell_2, \ell_3\}$.

The computation of the contribution of $\Gamma$ to $C'$ is essentially
the same as for $\kappa_1$.  The result is
\begin{equation*}
  -2\cdot \frac 14 \cdot \frac{12}6 \frac 1{\#\Aut(\Gamma)} \Pi_{1*}(\psi_2\psi_3\psi_4 \xi_{\Ga*}[1]) = -\frac 14 \cdot 24 \delta_{irr}.
\end{equation*}
The computation for $\Gamma'$ is also very similar.  The main
difference is that we should define $x := a_1 + a_2 + a_3$.  We find
that the contribution of $\Gamma'$ to $C'$ is
\begin{equation*}
  -2 \cdot \frac 14 \cdot \frac{12}6 \frac 1{\#\Aut(\Gamma')} \Pi_{1*}(\psi_2\psi_3\psi_4 \xi_{\Ga'*}[1]) = -\frac 14 \cdot 8 \delta_{irr}.
\end{equation*}

Combining the contributions, we obtain the relation
\begin{equation*}
  9 \kappa_1 + 3 \psi_1 = \delta_{irr} \in CH^*(\oM_{1, 1}),
\end{equation*}
and thus
\begin{equation*}
  \kappa_1 = \psi_1 = \frac 1{12} \delta_{irr} \in CH^*(\oM_{1, 1}).
\end{equation*}

\section*{Acknowledgments}

The authors would like to thank Aaron Pixton and Ravi Vakil for useful discussions, and Dimitri Zvonkine for pointing out a gap in an earlier version of the argument deducing the vanishing result for the tautological ring of $\Mgn$ from the theta relations. Our collaboration started during a visit by the first and third authors to Columbia University, where the second author was visiting on sabbatical. The authors would like to thank Columbia University for the hospitality.

\end{document}